\documentclass[11pt]{amsart}
\usepackage{graphicx}
\usepackage{amssymb}
\usepackage{amsmath}
\usepackage{amscd}
\usepackage{mathrsfs}
\usepackage{stmaryrd}
\usepackage{longtable}
\usepackage{txfonts}
\usepackage{wrapfig}
\usepackage{picins}
\usepackage{float}
\usepackage{xypic}
\usepackage{dsfont}
\usepackage{xcolor}
\usepackage{color}
\usepackage[colorlinks,linkcolor=blue,anchorcolor=blue,
citecolor=blue]{hyperref}
\usepackage{hyperref}
\allowdisplaybreaks

\newtheorem{theorem}{Theorem}[subsection]
\newtheorem{prop}[theorem]{Proposition}
\newtheorem{defn}[theorem]{Definition}

\newtheorem{coro}[theorem]{Corollary}
\newtheorem{rem}[theorem]{Remark}

\newtheorem{exam}[theorem]{Example}

\begin{document}
\setlength{\oddsidemargin}{0cm}
\setlength{\evensidemargin}{0cm}

\title{Cohomologies and Deformations of Generalized Left-symmetric Algebras}

\author{Runxuan Zhang}

\address{School of Mathematics and Statistics, Northeast Normal University,
 Changchun 130024, P.R. China}

\email{zhangrx728@nenu.edu.cn}

\date{\today}

\def\shorttitle{Cohomologies and Deformations of Generalized Left-symmetric Algebras}

\begin{abstract}
The purpose of this paper is to develop a cohomology and deformation theories for generalized left-symmetric algebras.
We introduce the notions of generalized left-symmetric cohomology and deformation. We also generalize a theorem of Dzhumadil'daev on
connections between the right-symmetric cohomology and  Chevalley-Eilenberg cohomology.
As an application, we obtain a factorization theorem  in left-symmetric superalgebras cohomology.
Finally, we obtain all complex simple
left-symmetric superalgebras of dimension 3 by the infinitesimal deformations of
a given left-symmetric superalgebras.
\end{abstract}



\maketitle
\baselineskip=16pt

\section{Introduction}
\setcounter{equation}{0}
\renewcommand{\theequation}
{1.\arabic{equation}}

\setcounter{theorem}{0}
\renewcommand{\thetheorem}
{1.\arabic{theorem}}

The supersymmetry theory is an active  and interesting   research object of mathematics and mathematical physics; its mathematical foundation
is the theory of Lie superalgebras  as well as
of Lie supergroups and of supermanifolds. The generalized Lie algebras,
which include Lie algebras and Lie superalgebras as special cases, have been described in \cite{Ree1960,Sch1979,Sch1983,Kle1985,Koc2008} from algebraic point of view and has been introduced into physics in 1978 (\cite{RW1978}).

Let $\Gamma$ be an abelian group and $k^{\ast}$ denote  the multiplicative group of a field $k$. Then a mapping
$\epsilon:\Gamma\times\Gamma\rightarrow k^{\ast}$ is called  a \textit{commutation factor} on $\Gamma$ if for all
$\alpha,\beta,\gamma\in \Gamma$,
\begin{eqnarray}
\epsilon(\alpha,\beta)\epsilon(\beta,\alpha)&=&1,\\
\epsilon(\alpha,\beta+\gamma)&=&\epsilon(\alpha,\beta)\epsilon(\alpha,\gamma),\\
\epsilon(\alpha+\beta,\gamma)&=&\epsilon(\alpha,\gamma)\epsilon(\beta,\gamma).
\end{eqnarray}
A $\Gamma$-graded
algebra $L=\oplus_{\gamma\in\Gamma}L_{\gamma}$ with the  multiplication $(x,y)\mapsto [x,y]$ is called a \textit{generalized Lie algebra} (or
\textit{$\epsilon$-Lie algebra}) if the following identities are satisfied:
\begin{eqnarray}
[x,y]+\epsilon(\alpha,\beta)[y,x]&=&0~~(\textrm{$\epsilon$-skew symmetry}),\\ \epsilon(\alpha,\gamma)[x,[y,z]]+\textrm{cyclic}&=&0~~(\textrm{$\epsilon$-Jacobi identity}),
\end{eqnarray}
for all $x\in
L_{\alpha}, y\in L_{\beta}, z\in L_{\gamma}, \alpha,\beta,\gamma\in \Gamma.$ On the other hand,
a non-associative
algebra $S$ with the multiplication $(x,y)\mapsto x\cdot y$ is called  a
\textit{left-symmetric algebra } (or \textit{pre-Lie algebra}) (\cite{Bau1999,Bur2006}) if the following identity is satisfied:
\begin{equation}
(x\cdot y)\cdot z-x\cdot(y\cdot z)=(y\cdot
x)\cdot z-y\cdot(x\cdot z)
\end{equation}
for all $x,y, z\in S.$
Since left-symmetric algebras are a kind of Lie admissible algebras, there are some close connections between left-symmetric algebra  and Lie algebra theories.
For example, if $G$ is a connected and simply connected Lie group with Lie algebra $\mathfrak{g}$, then the isomorphism classes of left-symmetric algebraic structures on
$\mathfrak{g}$ correspond to left-invariant flat affine structures on $G$.

The purpose of this paper is to introduce the notion of generalized left-symmetric algebras and to investigate their cohomology and deformation theories. The cohomology theory of Lie algebras as an important tool originated in the works of Cartan and was  developed   by Chevalley and Eilenberg \cite{CE1948}, Koszul \cite{Kos1950}, and Hochschild and Serre \cite{HS1953}.
Moreover, the cohomological constructions of Lie superalgebra and generalized Lie algebra first appeared in \cite{Lei1975} and \cite{MT1984} respectively.
(see \cite{Bin1986,FL1984,SZ1998,Lec1987} for the calculations and applications of  generalized Lie algebra cohomology.)
In 1999, Dzhumadil'daev \cite{Dzh1999} has defined cohomology groups
$H^{n}_{\textrm{rsym}}(S,M)$
for a right-symmetric algebra $S$ and an antisymmetric  module $M$ and moreover, it was proved that the right-symmetric cohomology can
be deduced from corresponding Lie algebra cohomology. In this paper, we will generalize the result of Dzhumadil'daev to the generalized left-symmetry algebra cohomology case.
In particular, we show that
\begin{equation}
H_{\textrm{lsym}}^{i+1}(S,M)\cong H^{i}_{\textrm{Lie}}(\mathfrak{g}_{S}, C^{1}(S,M)) \textrm{ for all } i>0,
\end{equation}
where $\mathfrak{g}_{s}$ denotes the associated $\epsilon$-Lie algebra of $S$.
Applying the Hochschild-Serre factorization theorem   in Lie superalgebras (\cite{Bin1986}), we obtain
a factorization theorem (Theorem 4.3) in left-symmetric superalgebras cohomology.

The formal
deformation theory was introduced  by Gerstenhaber (\cite{Ger}) for
associative algebras in 1960'. The fundamental results of Gerstenhaber's
theory connect deformation theory with the suitable cohomology
groups. Nowadays deformation-theoretic ideas penetrate most aspects
of both mathematics and physics and cut to the core of theoretical
and computational problems (\cite{BL1990,Fia1986,FdM2005,FP2007,GO1993,NR1967,Lev1970,HM2001,ZHB2011}).
In this paper, we also develop a formal deformation theory for generalized left-symmetric algebras.
As an application, we obtain all complex simple
left-symmetric superalgebras of dimension 3 by the infinitesimal deformations of
a given 3-dimensional left-symmetric superalgebras.

The present paper is organized as follows.
In Section 2, we introduce the notions of the generalized left-symmetric algebra and its bimodule. A few examples are presented.
We also define a generalized left-symmetric bimodule structure on the tensor product $M\otimes N$ of two bimodules $M$ and $N$. In particular, one can endow
$C^{i+1}(S,M)=\textrm{Hom}(\wedge_{\epsilon}^{i}S\otimes S,M),i\geq 0$ with an antisymmetric bimodule structure.
In Section 3, we  prove that   $\tilde{C}(S,M)=\oplus_{i\geq0}C^{i+1}(S,M)$
has the structure of an  $\epsilon$-pre-simplicial cochain complex. The notions of generalized
left-symmetric coboundary, cocycle and cohomology spaces are defined.
Section 4 deals with a generalization of  a theorem due to Dzhumadil'daev on
connections between the left-symmetric cohomology and  Chevalley-Eilenberg cohomology.
In Section 5, the  generalized left-symmetric deformations are introduced in the sense of Gerstenhaber.
We  determine all  complex simple left-symmetric superalgebras that can
be obtained by the infinitesimal deformations of a given simple
left-symmetric superalgebra of dimension 3. They are just
 all  simple left-symmetric superalgebras of dimension 3 in \cite{Zha2011} (or see \cite{ZB2012}).

Throughout this paper all vector spaces and (super)algebras  are assumed to be finite-dimensional over a field $k$ of characteristic zero.

\section{Generalized Left-symmetric Algebras and Bimodules}

\setcounter{equation}{0}
\renewcommand{\theequation}
{2.\arabic{equation}}

\setcounter{theorem}{0}
\renewcommand{\thetheorem}
{2.\arabic{theorem}}

We first introduce the notion of generalized left-symmetric algebra.

\begin{defn}{\rm
Let $\Gamma$ be an abelian group and $\epsilon$ be a commutation factor on $\Gamma$. A $\Gamma$-graded nonassociative
algebra $S=\oplus_{\gamma\in\Gamma}S_{\gamma}$ with the  multiplication $(x,y)\mapsto x\cdot y$ satisfying
\begin{equation}
S_\alpha\cdot S_\beta\subseteq S_{\alpha+\beta},
\end{equation}is called a \textit{generalized left-symmetric algebra} (or a
\textit{generalized left-symmetric algebra}) if the following identity is satisfied:
\begin{equation}
(x\cdot y)\cdot z-x\cdot(y\cdot z)=\epsilon(\alpha,\beta)((y\cdot
x)\cdot z-y\cdot(x\cdot z))
\end{equation}
for all $x\in
S_{\alpha}, y\in S_{\beta}, z\in S, \alpha,\beta\in \Gamma.$  $(x,y,z):=(x\cdot y)\cdot z-x\cdot(y\cdot z)$ is  called the \textit{associator} of
$S$.
}
\end{defn}

\begin{exam}{\rm 1. If we choose $\epsilon$ to be the trivial commutation factor, that is $\epsilon(\alpha,\beta)=1$ for all $\alpha,\beta\in \Gamma$. Then a $\Gamma$-graded generalized left-symmetric algebra is nothing but a $\Gamma$-graded left-symmetric algebra.

2. If $\Gamma=\textbf{Z}_{2}$ is the additive group of integers modulo 2 and $\epsilon(\alpha,\beta):=(-1)^{\alpha\beta}$ for all
 $\alpha,\beta\in \textbf{Z}_{2}$, then $\textbf{Z}_{2}$-graded generalized left-symmetric algebras are just left-symmetric superalgebras (see \cite{Zha2013} or \cite{KB2008}).

3. Let $V$ be a $\Gamma$-graded vector space and $\epsilon$ be an arbitrary  commutation factor on $\Gamma$. Then the sum of the subspaces of
all homogeneous  $k$-linear mappings from $V$ to $V$ with degree $\gamma$
\begin{equation}
\textrm{Hom}(V,V):=\oplus_{\gamma\in\Gamma}\textrm{Hom}(V,V)_{\gamma},
\end{equation}
 is a $\Gamma$-graded generalized left-symmetric algebra with respect to  the usual  composition of linear mappings. In fact, it is a $\Gamma$-graded associative algebra and
 denoted by $gl(V,\epsilon)$.
 }
\end{exam}

A generalized left-symmetric algebra $S$ is an \textit{$\epsilon$-Lie admissible} algebra, that is, for all $x\in S_{\alpha},y\in S_{\beta}$ and
$\alpha,\beta\in \Gamma$, the \textit{$\epsilon$-commutator }
\begin{equation}
[x,y]:=x\cdot y-\epsilon(\alpha,\beta)y\cdot x
\end{equation}
makes $S$ to become an $\epsilon$-Lie algebra, which
is called the \textit{associated} $\epsilon$-Lie algebra of $S$ and denoted by $\mathfrak{g}_{S}$.

\begin{rem}{\rm
We notice  that every generalized left-symmetric algebra $S$ has a natural $\textbf{Z}_{2}$-gradation. In fact, it follows from (1.1) that the mapping
\begin{equation}
\phi:\Gamma\rightarrow k^{\ast},~~ \alpha\mapsto \epsilon(\alpha,\alpha)
\end{equation}
is a homomorphism of groups and for  all $\alpha$ in $\Gamma$, we have $\epsilon(\alpha,\alpha)=\pm 1$.
Let us define
\begin{eqnarray}
\Gamma_{\bar{0}}&:=&\{\alpha\in \Gamma| \epsilon(\alpha,\alpha)=1\},\\
\Gamma_{\bar{1}}&:=&\{\alpha\in \Gamma| \epsilon(\alpha,\alpha)=-1\}.
\end{eqnarray} Then $\Gamma_{\bar{0}}=\textrm{Ker} \phi$ is a subgroup of $\Gamma$ and it follows that
the decomposition
\begin{equation}
S=S_{\bar{0}}\oplus S_{\bar{1}},
\end{equation}is a $\textbf{Z}_{2}$-gradation of $S$,
where $S_{i}:=\oplus_{\alpha\in \Gamma_{i}} S_{\alpha}$ for $i=\bar{0},\bar{1}$.
}\end{rem}

\begin{defn}{\rm
A $\Gamma$-graded vector space $M$ is said to be  a \textit{bimodule} over a generalized left-symmetric algebra $S$
if it is endowed with a left action $A\times M\rightarrow M,(x,m)\mapsto x\cdot m$ and a right action $M\times A\rightarrow M,(m,x)\mapsto m\cdot x$ such that
\begin{eqnarray}
(x\cdot y)\cdot m-x\cdot(y\cdot m)&=&\epsilon(\alpha,\beta)((y\cdot x)\cdot m-y\cdot(x\cdot m))\\
(x\cdot m)\cdot y-x\cdot(m\cdot y)&=&\epsilon(\alpha,\gamma)((m\cdot
x)\cdot y-m\cdot(x\cdot y))
\end{eqnarray} for all $x\in S_{\alpha},y\in S_{\beta},m\in M_{\gamma}$ and
$\alpha,\beta,\gamma\in \Gamma$.
}
\end{defn}

\begin{defn}{\rm
A generalized left-symmetric  $S$-bimodule $M$ is said to be \textit{antisymmetric} if the right action of $S$ is trivial, that is
$m\cdot x=0$ for all $x\in S$ and $m\in M$.

A generalized left-symmetric  $S$-bimodule $M$ is said to be \textit{special} if the left action of $S$ is associative, that is,
$(x\cdot y)\cdot m-x\cdot(y\cdot m)=0$  for all $x,y\in S$ and $m\in M$.
}
\end{defn}

\begin{exam}{\rm
Any  generalized left-symmetric algebra $S$ can be endowed with a natural $S$-bimodule structure,
$(x,m)\mapsto x\cdot m,(m,x)\mapsto m\cdot x$ for all $x,m\in S$. In this case, $S$ is called the \textit{regular} bimodule.
Moreover, it is easy to check that  the regular bimodule of $gl(V,\epsilon)$ is special.
}
\end{exam}

\begin{prop}
Let $S$ be a generalized left-symmetric algebra and $M$ be an $S$-bimodule. Then the graded space of $k$-linear maps
{\rm\begin{equation}
C^{1}(S,M):=\textrm{Hom}(S,M)=\oplus_{\gamma\in\Gamma} \textrm{Hom}_{\gamma}(S,M)
\end{equation}}
is an antisymmetric $S$-bimodule if we define the left action as follows:
\begin{equation}
(x\cdot f)(y):=x\cdot f(y)-\epsilon(\alpha,\gamma)f(x\cdot y)+\epsilon(\alpha,\gamma)f(x)\cdot y,
\end{equation}
where {\rm $x\in S_{\alpha},y\in S_{\beta},f\in \textrm{Hom}_{\gamma}(S,M)$} and $\alpha,\beta,\gamma\in \Gamma$.
\end{prop}

\begin{proof}
Let $f\in \textrm{Hom}_{\gamma}(S,M)$ and $x,y,z$ belong to $S_{\alpha},S_{\beta},S$ respectively, then
\begin{eqnarray*}
((x\cdot y)\cdot f)(z)&=&(x\cdot y)\cdot f(z)-\epsilon(\alpha+\beta,\gamma)f((x\cdot y)\cdot z)+\epsilon(\alpha+\beta,\gamma)f(x\cdot y)\cdot z,\\
((y\cdot x)\cdot f)(z)&=&(y\cdot x)\cdot f(z)-\epsilon(\alpha+\beta,\gamma)f((y\cdot x)\cdot z)+\epsilon(\alpha+\beta,\gamma)f(y\cdot x)\cdot z,\\
\end{eqnarray*} and
\begin{eqnarray*}
&&(x\cdot(y\cdot f))(z)\\
&=&x\cdot ((y\cdot f)(z))-\epsilon(\alpha,\beta+\gamma)(y\cdot f)(x\cdot z)+\epsilon(\alpha,\beta+\gamma)((y\cdot f)(x))\cdot z,\\
&=&x\cdot \big(y\cdot f(z)-\epsilon(\beta,\gamma)f(y\cdot z)+\epsilon(\beta,\gamma)f(y)\cdot z\big)-\\
&&\epsilon(\alpha,\beta+\gamma)(y\cdot f(x\cdot z)-\epsilon(\beta,\gamma)f(y\cdot (x\cdot z))+\epsilon(\beta,\gamma)f(y)\cdot(x\cdot z))+\\
&&\epsilon(\alpha,\beta+\gamma)(y\cdot f(x)-\epsilon(\beta,\gamma)f(y\cdot x)+\epsilon(\beta,\gamma)f(y)\cdot x)\cdot z.
\end{eqnarray*}
Similarly, we have
\begin{eqnarray*}
&&(y\cdot(x\cdot f))(z)\\
&=&y\cdot \big(x\cdot f(z)-\epsilon(\alpha,\gamma)f(x\cdot z)+\epsilon(\alpha,\gamma)f(x)\cdot z\big)-\\
&&\epsilon(\beta,\alpha+\gamma)(x\cdot f(y\cdot z)-\epsilon(\alpha,\gamma)f(x\cdot(y\cdot z))+\epsilon(\alpha,\gamma)f(x)\cdot(y\cdot z))+\\
&&\epsilon(\beta,\alpha+\gamma)(x\cdot f(y)-\epsilon(\alpha,\gamma)f(x\cdot y)+\epsilon(\alpha,\gamma)f(x)\cdot y)\cdot z.
\end{eqnarray*}
The direct computation shows that  $$((x\cdot y)\cdot f)(z)-(x\cdot(y\cdot f))(z)=\epsilon(\alpha,\beta)(((y\cdot x)\cdot f)(z)-(y\cdot(x\cdot f))(z)).$$
Hence, $C^{1}(S,M)$ is an antisymmetric $S$-bimodule.
\end{proof}

\begin{defn}{\rm If $\mathfrak{g}$ is an $\epsilon$-Lie algebra over a field $k$, then
a \textit{left $\mathfrak{g}$-module} is a $\Gamma$-graded vector space $M$  together with a multiplication
$\mathfrak{g}\times M\rightarrow M, (x,m)\mapsto [x,m]$ satisfying the axioms:

(1)~~ $(x,m)\mapsto [x,m]$ is linear in $x$ and in $m$;

(2)~~ $[[x,y],m]=[x,[y,m]]-\epsilon(\alpha,\beta)[y,[x,m]]$ for all $x\in \mathfrak{g}_{\alpha}$ and $y\in \mathfrak{g}_{\beta}$.
}
\end{defn}

A generalized left-symmetric $S$-bimodule $M$ can be endowed with an $\epsilon$-Lie module structure over $\mathfrak{g}_{S}$
by the  action
\begin{equation}
[x,m]:=x\cdot m-\epsilon(\beta,\alpha)m\cdot x
\end{equation}
for all $x\in \mathfrak{g}_{\alpha}$ and $m\in M_{\beta}$.
In this situation, we denote the associated  $\epsilon$-Lie module by $\mathfrak{M}$. Thus if one wants to define a generalized left-symmetric $S$-bimodule
structure over a $\Gamma$-graded vector space $M$, we should first give a  left $\epsilon$-Lie module on $M$ over $\mathfrak{g}_{S}$ and
endow it with a right action that satisfies the conditions in the definition of  $S$-bimodule.

\begin{prop}
Let  $M$ and $N$ be two generalized left-symmetric $S$-bimodules. Then the tensor product
$M\otimes N$ is an  $S$-bimodule if we define the left and  right actions as follows:
\begin{eqnarray}
x\cdot(m\otimes n)&:=&(x\cdot m-\epsilon(\alpha,\beta)m\cdot x)\otimes n+\epsilon(\alpha,\beta)m\otimes x\cdot n\\
(m\otimes n)\cdot x&:=& m\otimes n\cdot x,
\end{eqnarray}
where $x,m,n$ belong to $S_{\alpha},M_{\beta},N$ respectively.
\end{prop}

\begin{proof}
For arbitrary $x\in S_{\alpha},y\in S_{\beta},m\in M_{\gamma}$ and $n\in N_{\delta}$, it is clear that
$$(x\cdot (m\otimes n))\cdot y-x\cdot ((m\otimes n)\cdot y)=\epsilon(\alpha,\gamma+\delta)(((m\otimes n)\cdot x)\cdot y-(m\otimes n)\cdot(x\cdot y)).$$ Next it suffices to check that
$$(x\cdot y)\cdot (m\otimes n)-x\cdot(y\cdot(m\otimes n))=\epsilon(\alpha,\beta)((y\cdot x)\cdot(m\otimes n)-y\cdot(x\cdot(m\otimes n))).$$
In fact, \begin{eqnarray*}
(x\cdot y)\cdot(m\otimes n)&=&(x\cdot y)\cdot m\otimes n-\epsilon(\alpha+\beta,\gamma)m\cdot (x\cdot y)\otimes n+\epsilon(\alpha+\beta,\gamma)m\otimes (x\cdot y)\cdot n
\end{eqnarray*} and
\begin{eqnarray*}
&&x\cdot (y\cdot(m\otimes n))\\
&=&x\cdot (y\cdot m\otimes n-\epsilon(\beta,\gamma)m\cdot y\otimes n+\epsilon(\beta,\gamma)m\otimes y\cdot n)\\
&=&x\cdot (y\cdot m)\otimes n-\epsilon(\alpha,\beta+\gamma)(y\cdot m)\cdot x\otimes n+\epsilon(\alpha,\beta+\gamma)y\cdot m\otimes x\cdot n-\\
&&\epsilon(\beta,\gamma)(x\cdot (m\cdot y)\otimes n-\epsilon(\alpha,\beta+\gamma)(m\cdot y)\cdot x\otimes n+\epsilon(\alpha,\beta+\gamma)m\cdot y\otimes x\cdot n)+\\
&&\epsilon(\beta,\gamma)(x\cdot m\otimes y\cdot n-\epsilon(\alpha,\gamma)m\cdot x\otimes y\cdot n+\epsilon(\alpha,\gamma)m\otimes x\cdot (y\cdot n)).
\end{eqnarray*}
Similarly, we have
$$(y\cdot x)\cdot(m\otimes n)=(y\cdot x)\cdot m\otimes n-\epsilon(\alpha+\beta,\gamma)m\cdot (y\cdot x)\otimes n+\epsilon(\alpha+\beta,\gamma)m\otimes (y\cdot x)\cdot n$$ and
\begin{eqnarray*}
&&y\cdot (x\cdot(m\otimes n))\\
&=&y\cdot(x\cdot m)\otimes n-\epsilon(\beta,\alpha+\gamma)(x\cdot m)\cdot y\otimes n+\epsilon(\beta,\alpha+\gamma)x\cdot m\otimes y\cdot n-\\
&&\epsilon(\alpha,\gamma)(y\cdot(m\cdot x)\otimes n-\epsilon(\beta,\alpha+\gamma)(m\cdot x)\cdot y\otimes n+\epsilon(\beta,\alpha+\gamma)m\cdot x\otimes y\cdot n)+\\
&&\epsilon(\alpha,\gamma)(y\cdot m\otimes x\cdot n-\epsilon(\beta,\gamma)m\cdot y\otimes x\cdot n+\epsilon(\beta,\gamma)m\otimes y\cdot (x\cdot n)).
\end{eqnarray*}
The direct computation yields  the desired result.
\end{proof}

Let $V$ be a $\Gamma$-graded vector space, then the tensor algebra (see \cite{Sch1983})
\begin{equation}
T(V):=\oplus_{i\in\textbf{Z}}T_{i}(V)
\end{equation} is a $\textbf{Z}\times \Gamma$-graded associative algebra, where $T_{i}(V)=\{0\}$ for all $i<0$, $T_{0}(V)=k$ and $$T_{i}(V)=\underbrace{V\otimes\cdots\otimes V}_{i\textrm{ factors}}$$ for all $i>0$.
Let
$\mathcal{I}(V,\epsilon)$ denote the two-sided ideal of $T(V)$ generated by the elements of the form
\begin{equation}
m\otimes n+\epsilon(\alpha,\beta)n\otimes m,~~m\in V_{\alpha},n\in V_{\beta}.
\end{equation}
Then the quotient algebra $\wedge_{\epsilon}V:=T(V)/\mathcal{I}(V,\epsilon)$ is called the $\epsilon$-\textit{exterior algebra} of $V$ and the multiplication in $\wedge_{\epsilon}V$ is denoted by $\wedge_{\epsilon}$.
On the other hand, $\mathcal{I}(V,\epsilon)$ is a $\textbf{Z}\times \Gamma$-graded ideal of $T(V)$ and thus the quotient algebra $\wedge_{\epsilon}V$
inherits from $T(V)$ a canonical $\textbf{Z}\times \Gamma$-gradation. In particular, if we write
$\wedge_{\epsilon}^{i}V=V\wedge_{\epsilon}\cdots\wedge_{\epsilon}V (i\textrm{ factors})$ for the canonical image of
$T_{i}(V)$ in $\wedge_{\epsilon}V$, then
\begin{equation}
\wedge_{\epsilon}V=\oplus_{i\in \textbf{Z}}\wedge_{\epsilon}^{i}V,
\end{equation}
where $\wedge_{\epsilon}^{i}V={0}$ for all $i<0$, $\wedge_{\epsilon}^{0}(V)=k$ and $\wedge_{\epsilon}^{1}(V)=V$.
Define
\begin{equation}
C^{i+1}(S,M):=\textrm{Hom} ( \wedge_{\epsilon}^{i} S\otimes S,M)=
\oplus_{\gamma\in\Gamma}\textrm{Hom}_{\gamma}( \wedge_{\epsilon}^{i} S\otimes S,M), i\geq0.
\end{equation}

\begin{prop} For each $i\geq 0$,
$C^{i+1}(S,M)$ is an antisymmetric $S$-bimodule if we endow it with the following left action:
\begin{equation}
S\times C^{i+1}(S,M)\rightarrow C^{i+1}(S,M), (x,f)\mapsto x\cdot f,
\end{equation}
where {\rm $f\in \textrm{Hom}_{\beta}(\wedge_{\epsilon}^{i} S\otimes S,M)$}, $x\in S_{\alpha}$, $x_{s}\in S_{\alpha_{s}} (s=1,\cdots,i+1)$ and
\begin{eqnarray*}
&&(x\cdot f)(x_{1},\cdots,x_{i},x_{i+1})\\
&=& x\cdot f(x_{1},\cdots,x_{i},x_{i+1})-\epsilon(\alpha,\beta+\alpha_1+\cdots+\alpha_i)f(x_{1},\cdots,x_{i},x\cdot x_{i+1})\\
&& +\epsilon(\alpha,\beta+\alpha_1+\cdots+\alpha_i) f(x_{1},\cdots,x_{i},x)\cdot x_{i+1}\\
&& -\sum_{s=1}^{i}\epsilon(\alpha,\beta+\alpha_{1}+\cdots+\alpha_{s-1}) f (x_{1},\cdots,x_{s-1},[x,x_{s}],\cdots,x_{i},x_{i+1}).\end{eqnarray*}
\end{prop}

\begin{proof}
Let $\mathfrak{g}_{S}$  denote the associated $\epsilon$-Lie algebra.
Note that for each $i\geq 0$, there is an isomorphism $\eta$ which is homogeneous of $\Gamma$-degree zero of vector spaces:
$$\eta:  C^{i}(\mathfrak{g}_{S},k)\otimes C^{1}(S,M)\rightarrow C^{i+1}(S,M),~~g\otimes h\mapsto \eta(g\otimes h),$$
where
\begin{equation}
C^{i}(\mathfrak{g}_{S},k):=\textrm{Hom} (\wedge_{\epsilon}^{i} \mathfrak{g}_{S},k)=
\oplus_{\gamma\in\Gamma}\textrm{Hom}_{\gamma}(\wedge_{\epsilon}^{i} \mathfrak{g}_{S},k), i\geq0
\end{equation} is an $\epsilon$-Lie module of $\mathfrak{g}_{S}$ (see \cite{SZ1998} for details) and for each $g\in \textrm{Hom}_{\gamma}(\wedge_{\epsilon}^{i} \mathfrak{g}_{S},k),$ $ h\in \textrm{Hom}_\delta(S,M)$,
$$\eta(g\otimes h)(x_{1},\cdots,x_{i},x_{i+1}):=\epsilon(\delta,\alpha_1+\cdots+\alpha_{i})g(x_{1},\cdots,x_{i})h(x_{i+1}).$$
Thus the antisymmetric generalized left-symmetric bimodule structure on $C^{i}(\mathfrak{g}_{S},k)$ $\otimes C^{1}(S,M)$ will give an antisymmetric
generalized left-symmetric bimodule structure on $C^{i+1}(S,M)$ by the transformation  of $\eta$.
On the other hand, the antisymmetric $S$-bimodule structure on $ C^{1}(S,M)$ is constructed in Proposition 2.7 and the
$\epsilon$-Lie module structure on $C^{i}(\mathfrak{g}_{S},k)$ over $\mathfrak{g}_{S}$ is well known. Hence for
$h\in \textrm{Hom}_{\delta}(S,M)$ and $g\in \textrm{Hom}_{\gamma}(\wedge_{\epsilon}^{i} \mathfrak{g}_{S},k)$, we define a left action of $S$ on
$C^{i}(\mathfrak{g}_{S},k)\otimes C^{1}(S,M)$ by:
\begin{eqnarray*}
&&\eta(x\cdot (g\otimes h))(x_{1},\cdots,x_{i},x_{i+1})\\
&=&\eta([x,g]\otimes h +\epsilon(\alpha,\gamma)g\otimes x\cdot h)(x_{1},\cdots,x_{i},x_{i+1})\\
&=&-\sum_{s=1}^{i} \epsilon(\alpha,\gamma+\alpha_1+\cdots+\alpha_{s-1})\epsilon(\delta,\alpha_1+\cdots+\alpha_i) g(x_{1},\cdots,[x,x_{s}],\cdots,x_{i})h(x_{i+1})+\\
&&\epsilon(\alpha,\gamma)\epsilon(\alpha+\delta,\alpha_1+\cdots+\alpha_i)g(x_{1},\cdots,x_{i})\times\\
&&(x\cdot h(x_{i+1})-\epsilon(\alpha,\delta)h(x\cdot x_{i+1})+\epsilon(\alpha,\delta)h(x)\cdot x_{i+1}).
\end{eqnarray*}
The direct  calculations show that $\eta(x\cdot (g\otimes h))=x\cdot(\eta(g\otimes h)).$ Therefore  $(x,f)\mapsto x\cdot f$
is an antisymmetric $S$-bimodule.
\end{proof}

\section{Cohomology of Generalized Left-symmetric Algebras}

\setcounter{equation}{0}
\renewcommand{\theequation}
{3.\arabic{equation}}

\setcounter{theorem}{0}
\renewcommand{\thetheorem}
{3.\arabic{theorem}}

Let $(S,\Gamma,\epsilon)$ be a generalized left-symmetric algebra and $M$ be an $S$-bimodule.
We define
$\widetilde{C}(S,M):=\oplus_{i\geq 0}C^{i+1}(S,M),$ where $C^{i+1}(S,M)$ is defined as in (2.19).
Now we present an $\epsilon$-pre-simplicial cochain complex on $\widetilde{C}(S,M)$.
We  define homogeneous  linear mappings $D_{t}$  of degree zero with respect to the $\Gamma$-gradation on $\widetilde{C}(S,M)$ by
$D_{t}:C^{i}(S,M)\rightarrow C^{i+1}(S,M), (t=1,2,3,\cdots.)$
\begin{eqnarray}
&&(D_{t}f)(x_{1},\cdots,x_{i}, x_{i+1})\nonumber\\
&&=\epsilon(\beta+\alpha_{1}+\cdots+\alpha_{t-1},\alpha_{t})x_{t}\cdot f(x_{1},\cdots,\hat{x}_{t},\cdots,x_{i}, x_{i+1})\\
&&-\epsilon(\alpha_{t},\alpha_{t+1}+\cdots+\alpha_{i})f(x_{1},\cdots,\hat{x}_{t},\cdots,x_{i},x_t\cdot  x_{i+1})\nonumber\\
&&+\epsilon(\alpha_{t},\alpha_{t+1}+\cdots+\alpha_{i})f(x_{1},\cdots,\hat{x}_{t},\cdots,x_{i},x_t)\cdot x_{i+1}\nonumber\\
&&-\sum _{t< j}\epsilon(\alpha_{t},\alpha_{t+1}
+\cdots+\alpha_{j-1})f(x_{1},\cdots,\hat{x}_{t},\cdots,x_{j-1},[x_{t},x_{j}],\cdots,x_{i}, x_{i+1})\nonumber\end{eqnarray}
for $t\leq i, i\geq1$ and
\begin{equation}
D_tf=0 \textrm{ ~~for } t>i,
\end{equation}
where
 $f \in \textrm{Hom}_{\beta}(\wedge_{\epsilon}^{i-1} S\otimes S,M)$, $x_{s}\in S_{\alpha_{s}} (s=1,\cdots,i+1)$.
Here $\hat{x}$ means that the element $x$ is omitted.

\begin{prop}
For all $1\leq s<t$, we have $D_{t}D_{s}=D_{s}D_{t-1}$. That is, the set of $D_{i}$ endows
$\widetilde{C}(S,M)=\oplus_{i\geq 1}C^{i}(S,M)$ with an $\epsilon$-pre-simplicial structure.
\end{prop}

\begin{proof}
With the notations as above, for $s<t, t>i,$ $D_{t}D_{s}=D_{s}D_{t-1}=0.$
For $1\leq s<t\leq i,$ by (3.1), we have
\begin{eqnarray*}
&&(D_{t}D_{s}f)(x_{1},\cdots,x_{i}, x_{i+1})\\
&&=\epsilon(\beta+\alpha_{1}+\cdots+\alpha_{t-1},\alpha_{t})x_{t}\cdot(D_{s}f)(x_{1},\cdots,\hat{x}_{t},\cdots,x_{i}, x_{i+1})-\\
&&\epsilon(\alpha_{t},\alpha_{t+1}+\cdots+\alpha_{i})(D_{s}f)(x_{1},\cdots,\hat{x}_{t},\cdots,x_{i},x_t\cdot x_{i+1})+\\
&&\epsilon(\alpha_{t},\alpha_{t+1}+\cdots+\alpha_{i})(D_{s}f)(x_{1},\cdots,\hat{x}_{t},\cdots,x_{i},x_t)\cdot x_{i+1}-\\
&&\sum _{t< j}\epsilon(\alpha_{t},\alpha_{t+1}
+\cdots+\alpha_{j-1})(D_{s}f)(x_{1},\cdots,\hat{x}_{t},\cdots,x_{j-1},[x_{t},x_{j}],\cdots,x_{i},x_{i+1}),\end{eqnarray*}
where
\begin{eqnarray*}
&&x_t\cdot (D_{s}f)(x_{1},\cdots,\hat{x}_{t},\cdots,x_{i},x_{i+1})\\
&=&\epsilon(\beta+\alpha_{1}+\cdots+\alpha_{s-1},\alpha_{s})x_t\cdot (x_{s}\cdot f(x_{1},\cdots,\hat{x}_{s},\cdots,\hat{x}_{t},\cdots,x_{i},x_{i+1}))\\
&&
-\epsilon(\alpha_{s},\alpha_{s+1}+\cdots+\hat\alpha_{t}+\cdots+\alpha_{i})x_t\cdot f(x_{1},\cdots,\hat{x}_{s},\cdots,\hat{x}_{t},\cdots,x_{i},x_s\cdot x_{i+1})\\
&&+\epsilon(\alpha_{s},\alpha_{s+1}+\cdots+\hat\alpha_{t}+\cdots+\alpha_{i})x_t\cdot
(f(x_{1},\cdots,\hat{x}_{s},\cdots,\hat{x}_{t},\cdots,x_{i},x_s)\cdot x_{i+1})\\
&&-\sum _{s< j<t}\epsilon(\alpha_{s},\alpha_{s+1}
+\cdots+\alpha_{j-1})x_t\cdot f(x_{1},\cdots,\hat{x}_{s},\cdots,[x_{s},x_{j}],\cdots,\hat{x}_{t},\cdots,x_{i}, x_{i+1})\\
&&-\sum _{s<t<j}\epsilon(\alpha_{s},\alpha_{s+1}+\cdots+\hat\alpha_{t}+\cdots+\alpha_{j-1})x_t\cdot f(x_{1},\cdots,\hat{x}_{s},\cdots,\hat{x}_{t},\cdots,[x_{s},x_{j}],\cdots, x_{i+1}),\end{eqnarray*}
\begin{eqnarray*}
&&(D_{s}f)(x_{1},\cdots,\hat{x}_{t},\cdots,x_{i},x_t\cdot x_{i+1})\\
&=&\epsilon(\beta+\alpha_{1}+\cdots+\alpha_{s-1},\alpha_{s})x_{s}\cdot f(x_{1},\cdots,\hat{x}_{s},\cdots,\hat{x}_{t},\cdots,x_{i},x_t\cdot x_{i+1})\\
&&-
\epsilon(\alpha_{s},\alpha_{s+1}+\cdots+\hat\alpha_{t}+\cdots+\alpha_{i})f(x_{1},\cdots,\hat{x}_{s},\cdots,\hat{x}_{t},\cdots,x_{i},x_s\cdot(x_t\cdot x_{i+1}))\\
&&+\epsilon(\alpha_{s},\alpha_{s+1}+\cdots+\hat\alpha_{t}+\cdots+\alpha_{i})
f(x_{1},\cdots,\hat{x}_{s},\cdots,\hat{x}_{t},\cdots,x_{i},x_s)\cdot(x_t\cdot x_{i+1})\\
&&-\sum _{s< j<t}\epsilon(\alpha_{s},\alpha_{s+1}
+\cdots+\alpha_{j-1})f(x_{1},\cdots,\hat{x}_{s},\cdots,[x_{s},x_{j}],\cdots,\hat{x}_{t},\cdots,x_{i},x_t\cdot x_{i+1})\\
&&-\sum _{s<t< j}\epsilon(\alpha_{s},\alpha_{s+1}+\cdots+\hat\alpha_{t}+\cdots+\alpha_{j-1})f(x_{1},\cdots,\hat{x}_{s},\cdots,\hat{x}_{t},\cdots,[x_{s},x_{j}],\cdots,x_t\cdot x_{i+1}),\end{eqnarray*}
\begin{eqnarray*}
&&(D_{s}f)(x_{1},\cdots,\hat{x}_{t},\cdots,x_{i},x_t)\cdot x_{i+1}\\
&=&\epsilon(\beta+\alpha_{1}+\cdots+\alpha_{s-1},\alpha_{s})(x_{s}\cdot f(x_{1},\cdots,\hat{x}_{s},\cdots,\hat{x}_{t},\cdots,x_{i},x_t))\cdot x_{i+1}\\
&&-
\epsilon(\alpha_{s},\alpha_{s+1}+\cdots+\hat\alpha_{t}+\cdots+\alpha_{i})f(x_{1},\cdots,\hat{x}_{s},\cdots,\hat{x}_{t},\cdots,x_{i},x_s\cdot x_t)\cdot x_{i+1}\\
&&+\epsilon(\alpha_{s},\alpha_{s+1}+\cdots+\hat\alpha_{t}+\cdots+\alpha_{i})
(f(x_{1},\cdots,\hat{x}_{s},\cdots,\hat{x}_{t},\cdots,x_{i},x_s)\cdot x_t)\cdot x_{i+1}\\
&&-\sum _{s< j<t}\epsilon(\alpha_{s},\alpha_{s+1}
+\cdots+\alpha_{j-1})f(x_{1},\cdots,\hat{x}_{s},\cdots,[x_{s},x_{j}],\cdots,\hat{x}_{t},\cdots,x_{i},x_t)\cdot x_{i+1}\\
&&-\sum _{s<t< j}\epsilon(\alpha_{s},\alpha_{s+1}+\cdots+\hat\alpha_{t}+\cdots+\alpha_{j-1})f(x_{1},\cdots,\hat{x}_{s},\cdots,\hat{x}_{t},\cdots,[x_{s},x_{j}],\cdots,x_t)\cdot x_{i+1},\end{eqnarray*}
\begin{eqnarray*}
&&\sum _{t< j}(D_{s}f)(x_{1},\cdots,\hat{x}_{t},\cdots,x_{j-1},[x_{t},x_{j}],\cdots,x_{i},x_{i+1})\\
&=&\sum _{t< j}\epsilon(\beta+\alpha_{1}+\cdots+\alpha_{s-1},\alpha_{s})x_{s}\cdot f(x_{1},\cdots,\hat{x}_{s},\cdots,\hat{x}_{t},\cdots,[x_{t},x_j],\cdots,x_i, x_{i+1})\\
&&
-\sum _{t< j}\epsilon(\alpha_{s},\alpha_{s+1}+\cdots+\alpha_{i})f(x_{1},\cdots,\hat{x}_{s},\cdots,\hat{x}_{t},\cdots,[x_{t},x_{j}],\cdots,x_{i},x_s\cdot x_{i+1})\\
&&+\sum _{t< j}\epsilon(\alpha_{s},\alpha_{s+1}+\cdots+\alpha_{i})
f(x_{1},\cdots,\hat{x}_{s},\cdots,\hat{x}_{t},\cdots,[x_{t},x_{j}],\cdots,x_{i},x_s)\cdot x_{i+1}\\
&&-\sum _{s< j_1<t<j}\epsilon(\alpha_{s},\alpha_{s+1}
+\cdots+\alpha_{j_1-1})f(x_{1},\cdots,\hat{x}_{s},\cdots,[x_{s},x_{j_1}],\cdots,\hat{x}_{t},\cdots,[x_{t},x_{j}],\cdots, x_{i+1})\\
&&-\sum _{s< j_1,t<j<j_1}\epsilon(\alpha_{s},\alpha_{s+1}
+\cdots+\alpha_{j_1-1})f(x_{1},\cdots,\hat{x}_{s},\cdots,\hat{x}_{t},\cdots,[x_{t},x_{j}],\cdots,[x_{s},x_{j_1}],\cdots, x_{i+1})\\
&&-\sum _{t< j}\epsilon(\alpha_{s},\alpha_{s+1}+\cdots+\hat\alpha_{t}+\cdots+\alpha_{j-1})f(x_{1},\cdots,\hat{x}_{s},\cdots,\hat{x}_{t},\cdots,[x_{s},[x_{t},x_{j}]], \cdots, x_{i+1})\\
&&-\sum _{s<j_1,t< j_1< j}\epsilon(\alpha_{s},\alpha_{s+1}+\cdots+\hat\alpha_{t}+\cdots+\alpha_{j_1-1})\\
&&~~~~~~\times f(x_{1},\cdots,\hat{x}_{s},\cdots,\hat{x}_{t},\cdots,[x_{s},x_{j_1}],\cdots,[x_{t},x_{j}], \cdots,x_{i+1}).
\end{eqnarray*}

In a similar way, we have
\begin{eqnarray*}
&&(D_{s}D_{t-1}f)(x_{1},\cdots,x_{i},x_{i+1})\\
&=&\epsilon(\beta+\alpha_{1}+\cdots+\alpha_{s-1},\alpha_{s})x_{s}\cdot(D_{t-1}f)(x_{1},\cdots,\hat{x}_{s},\cdots,x_{i},x_{i+1})\\
&&-
\epsilon(\alpha_{s},\alpha_{s+1}+\cdots+\alpha_{i})(D_{t-1}f)(x_{1},\cdots,\hat{x}_{s},\cdots,x_{i},x_s\cdot x_{i+1})\\
&&+\epsilon(\alpha_{s},\alpha_{s+1}
+\cdots+\alpha_{i})
(D_{t-1}f)(x_{1},\cdots,\hat{x}_{s},\cdots,x_{i},x_s)\cdot x_{i+1}\\
&&-\sum _{s< j}\epsilon(\alpha_{s},\alpha_{s+1}
+\cdots+\alpha_{j-1})
(D_{t-1}f)(x_{1},\cdots,\hat{x}_{s},\cdots,x_{j-1},[x_{s},x_{j}],\cdots,x_{i},x_{i+1}),\end{eqnarray*}
where
\begin{eqnarray*}
&&x_s\cdot (D_{t-1}f)(x_{1},\cdots,\hat{x}_{s},\cdots,x_{i},x_{i+1})\\
&=&\epsilon(\beta+\alpha_{1}+\cdots+\hat\alpha_{s}+\cdots+\alpha_{t-1},\alpha_{t})x_s\cdot (x_{t}\cdot f(x_{1},\cdots,\hat{x}_{s},\cdots,\hat{x}_{t},\cdots,x_{i},x_{i+1}))\\
&&
-\epsilon(\alpha_{t},\alpha_{t+1}+\cdots+\alpha_{i})x_s\cdot f(x_{1},\cdots,\hat{x}_{s},\cdots,\hat{x}_{t},\cdots,x_{i},x_t\cdot x_{i+1})\\
&&+\epsilon(\alpha_{t},\alpha_{t+1}+\cdots+\alpha_{i})x_s\cdot
(f(x_{1},\cdots,\hat{x}_{s},\cdots,\hat{x}_{t},\cdots,x_{i},x_t)\cdot x_{i+1})\\
&&-\sum _{t<j}\epsilon(\alpha_{t},\alpha_{t+1}+\cdots+\alpha_{j-1})x_s\cdot f(x_{1},\cdots,\hat{x}_{s},\cdots,\hat{x}_{t},\cdots,[x_{t},x_{j}],\cdots,x_i, x_{i+1}),\end{eqnarray*}
\begin{eqnarray*}
&&(D_{t-1}f)(x_{1},\cdots,\hat{x}_{s},\cdots,x_{i},x_s\cdot x_{i+1})\\
&=&\epsilon(\beta+\alpha_{1}+\cdots+\hat\alpha_{s}+\cdots+\alpha_{t-1},\alpha_{t})x_{t}\cdot f(x_{1},\cdots,\hat{x}_{s},\cdots,\hat{x}_{t},\cdots,x_{i},x_s\cdot x_{i+1})\\
&&-
\epsilon(\alpha_{t},\alpha_{t+1}+\cdots+\alpha_{i})f(x_{1},\cdots,\hat{x}_{s},\cdots,\hat{x}_{t},\cdots,x_{i},x_t\cdot(x_s\cdot x_{i+1}))\\
&&+\epsilon(\alpha_{t},\alpha_{t+1}+\cdots+\alpha_{i})
f(x_{1},\cdots,\hat{x}_{s},\cdots,\hat{x}_{t},\cdots,x_{i},x_t)\cdot(x_s\cdot x_{i+1})\\
&&-\sum _{t< j}\epsilon(\alpha_{t},\alpha_{t+1}+\cdots+\alpha_{j-1})f(x_{1},\cdots,\hat{x}_{s},\cdots,\hat{x}_{t},\cdots,[x_{t},x_{j}],\cdots,x_i,x_s\cdot x_{i+1}),\end{eqnarray*}
\begin{eqnarray*}
&&(D_{t-1}f)(x_{1},\cdots,\hat{x}_{s},\cdots,x_{i},x_s)\cdot x_{i+1}\\
&=&\epsilon(\beta+\alpha_{1}+\cdots+\hat\alpha_{s}+\cdots+\alpha_{t-1},\alpha_{t})(x_{t}\cdot f(x_{1},\cdots,\hat{x}_{s},\cdots,\hat{x}_{t},\cdots,x_{i},x_s))\cdot x_{i+1}\\
&&-
\epsilon(\alpha_{t},\alpha_{t+1}+\cdots+\alpha_{i})f(x_{1},\cdots,\hat{x}_{s},\cdots,\hat{x}_{t},\cdots,x_{i},x_t\cdot x_s)\cdot x_{i+1}\\
&&+\epsilon(\alpha_{t},\alpha_{t+1}+\cdots+\alpha_{i})
(f(x_{1},\cdots,\hat{x}_{s},\cdots,\hat{x}_{t},\cdots,x_{i},x_t)\cdot x_s)\cdot x_{i+1}\\
&&-\sum _{t< j}\epsilon(\alpha_{t},\alpha_{t+1}+\cdots+\alpha_{j-1})f(x_{1},\cdots,\hat{x}_{s},\cdots,\hat{x}_{t},\cdots,[x_{t},x_{j}],\cdots,x_i,x_s)\cdot x_{i+1}.\end{eqnarray*}
We represent $\sum _{s< j}
(D_{t-1}f)(x_{1},\cdots,\hat{x}_{s},\cdots,x_{j-1},[x_{s},x_{j}],\cdots,x_{i},x_{i+1})$ as a sum of $X_1, X_2$ and $X_3$,
where
\begin{eqnarray*}
X_1
&=&\sum _{s< j<t}\epsilon(\beta+\alpha_{1}+\cdots+\alpha_{t-1},\alpha_{t})
x_{t}\cdot f(x_{1},\cdots,\hat{x}_{s},\cdots,[x_{s},x_{j}],\cdots,\hat{x}_{t},\cdots,x_{i},x_{i+1})\\
&&-\sum _{s< j<t}\epsilon(\alpha_{t},\alpha_{t+1}+\cdots+\alpha_{i})
f(x_{1},\cdots,\hat{x}_{s},\cdots,[x_{s},x_{j}],\cdots,\hat{x}_{t},\cdots,x_{i},x_t\cdot x_{i+1})\\
&&+\sum _{s< j<t}\epsilon(\alpha_{t},\alpha_{t+1}
+\cdots+\alpha_{i})f(x_{1},\cdots,\hat{x}_{s},\cdots,[x_{s},x_{j}],\cdots,\hat{x}_{t},\cdots,x_{i},x_{t})\cdot x_{i+1}\\
&&-\sum _{s<j<t<j_1}\epsilon(\alpha_{t},\alpha_{t+1}
+\cdots+\alpha_{j_1-1})\\
&&~~~~~~\times f(x_{1},\cdots,\hat{x}_{s},\cdots,[x_{s},x_{j}],\cdots,\hat{x}_{t},\cdots,[x_{t},x_{j_1}],\cdots,x_{i},x_{i+1}),\end{eqnarray*}
\begin{eqnarray*}
X_2&=&
\epsilon(\beta+\alpha_{1}+\cdots+\hat\alpha_{s}+\cdots+\alpha_{t-1},\alpha_{s}+\alpha_{t})[x_{s},x_t]\cdot f(x_{1},\cdots,\hat{x}_{s},\cdots,\hat{x}_{t},\cdots, x_{i+1})\\
&&
-\epsilon(\alpha_{s}+\alpha_{t},\alpha_{t+1}+\cdots+\alpha_{i})f(x_{1},\cdots,\hat{x}_{s},\cdots,\hat{x}_{t},\cdots,x_{i},[x_{s},x_{t}]\cdot x_{i+1})\\
&&+\epsilon(\alpha_{s}+\alpha_{t},\alpha_{t+1}+\cdots+\alpha_{i})
f(x_{1},\cdots,\hat{x}_{s},\cdots,\hat{x}_{t},\cdots,x_{i},[x_{s},x_{t}])\cdot x_{i+1}\\
&&-\sum_{t<j}\epsilon(\alpha_{s}+\alpha_{t},\alpha_{t+1}
+\cdots+\alpha_{j-1})f(x_{1},\cdots,\hat{x}_{s},\cdots,\hat{x}_{t},\cdots,[[x_s,x_{t}],x_{j}],\cdots, x_{i+1}),\end{eqnarray*}
\begin{eqnarray*}
X_3&=&\sum _{s<t< j}\epsilon(\beta+\alpha_{1}+\cdots+\hat\alpha_{s}+\cdots+\alpha_{t-1},\alpha_{t})x_{t}\cdot f(x_{1},\cdots,\hat{x}_{s},\cdots,\hat{x}_{t},\cdots,[x_{s},x_j],\cdots, x_{i+1})\\
&&
-\sum _{s<t< j}\epsilon(\alpha_{t},\alpha_{t+1}+\cdots+\alpha_{i}+\alpha_{s})f(x_{1},\cdots,\hat{x}_{s},\cdots,\hat{x}_{t},\cdots,[x_{s},x_{j}],\cdots,x_{i},x_t\cdot x_{i+1})\\
&&+\sum _{s<t< j}\epsilon(\alpha_{t},\alpha_{t+1}+\cdots+\alpha_{i}+\alpha_{s})
f(x_{1},\cdots,\hat{x}_{s},\cdots,\hat{x}_{t},\cdots,[x_{s},x_{j}],\cdots,x_{i},x_t)\cdot x_{i+1}\\
&&-\sum _{ t<j_1<j}\epsilon(\alpha_{t},\alpha_{t+1}
+\cdots+\alpha_{j_1-1})f(x_{1},\cdots,\hat{x}_{s},\cdots,\hat{x}_{t},\cdots,[x_{t},x_{j_1}],\cdots,[x_{s},x_{j}],\cdots, x_{i+1})\\
&&-\sum _{t< j}\epsilon(\alpha_{t},\alpha_{t+1}+\cdots+\alpha_{j-1})f(x_{1},\cdots,\hat{x}_{s},\cdots,\hat{x}_{t},\cdots,[x_{t},[x_{s},x_{j}]], \cdots, x_{i+1})\\
&&-\sum _{t<j<j_1}\epsilon(\alpha_{t},\alpha_{t+1}
+\cdots+\alpha_{j_1-1}+\alpha_{s})\\
&&~~~~~~\times f(x_{1},\cdots,\hat{x}_{s},\cdots,\hat{x}_{t},\cdots,[x_{s},x_{j}],\cdots,[x_{t},x_{j_1}],\cdots, x_{i+1}).
\end{eqnarray*}

Using the identities in the definitions of a generalized left-symmetric algebra and a bimodule and by  tedious calculations, we obtain  $D_{t}D_{s}=D_{s}D_{t-1}, s<t$.
\end{proof}

\begin{coro}
Let $d=-\sum_{i}(-1)^{i}D_{i}$, then $d^{2}=0$.
\end{coro}

Define $C^{0}(S,M):=\{m\in M| (ab)m=a(bm) \textrm{ for all } a,b\in S\}$, it is easy to check that
$C^{0}(S,M)$ is an $S$-submodule of $M$. By Corollary 3.2, the   cochain complex $\widetilde{C}(S,M)=\oplus_{i> 0}C^{i}(S,M)$ with its coboundary operator $d$
can be extended to a new cochain complex $C(S,M)=\oplus_{i\geq0}C^{i}(S,M)$ if we define $d:C^{0}(S,M)\rightarrow C^{1}(S,M)$ by
$$d(f)(x)=\epsilon(\beta,\alpha)x\cdot f-f\cdot x$$ for $x\in S_{\alpha}$ and $f\in C^{0}(S,M)\cap M_{\beta}$. Hence

\begin{prop} The space $C(S,M)=\oplus_{i\geq 0}C^{i}(S,M)$ is a cochain complex under the following  coboundary operator $d$ of degree zero
with respect to the $\Gamma$-gradation of $C(S,M)$,
$$d(f)(x)=\epsilon(\beta,\alpha)x\cdot f-f\cdot x,$$
where $x\in S_{\alpha}$,  $f\in C^{0}(S,M)\cap M_{\beta}$ and
\begin{eqnarray*}
&&(df)(x_{1},\cdots,x_{i},x_{i+1})\\
&=&-\sum_{t=1}^{i}(-1)^{t}\epsilon(\beta+\alpha_{1}+\cdots+\alpha_{t-1},\alpha_{t})
x_{t}\cdot f(x_{1},\cdots,\hat{x}_{t},\cdots,x_{i},x_{i+1})\\
&&+\sum_{t=1}^{i}(-1)^{t}
\epsilon(\alpha_{t},\alpha_{t+1}+\cdots+\alpha_{i})f(x_{1},\cdots,\hat{x}_{t},\cdots,x_{i},x_t\cdot x_{i+1})\\
&&-\sum_{t=1}^{i}(-1)^{t}\epsilon(\alpha_{t},\alpha_{t+1}
+\cdots+\alpha_{i})f(x_{1},\cdots,\hat{x}_{t},\cdots,x_{i},x_t)\cdot x_{i+1}\\
&&+\sum _{1\leq t< j\leq i}(-1)^{t}\epsilon(\alpha_{t},\alpha_{t+1}
+\cdots+\alpha_{j-1})
f(x_{1},\cdots,\hat{x}_{t},\cdots,x_{j-1},[x_{t},x_{j}],\cdots,x_{i},x_{i+1}),\end{eqnarray*}
where $f \in \textrm{Hom}_{\beta}(\wedge_{\epsilon}^{i-1} S\otimes S,M)$, $x_{s}\in S_{\alpha_{s}} (s=1,\cdots,i+1)$ and  $i\geq1$.
\end{prop}

We conclude this section with introducing the concept of the cohomology spaces of a generalized left-symmetric algebra $S$ with coefficients in an $S$-bimodule $M$.

\begin{defn}{\rm
Let $Z(S,M)=\oplus_{i\geq 0}Z^{i}(S,M)$ with
$$Z^{i}(S,M):=\{f\in C^{i}(S,M)|d(f)=0\}$$ denote the spaces of generalized left-symmetric cocycles and
$B(S,M)=\oplus_{i\geq 0}B^{i}(S,M)$ with
$$B^{i}(S,M):=\{d(g)|g\in C^{i-1}(S,M)\}$$
denote the spaces of generalized left-symmetric coboundaries. Then
$H(S,M)=\oplus_{i\geq 0}H^{i}(S,M)$ with
$$H^{i}(S,M):=Z^{i}(S,M)/B^{i}(S,M)$$
are called the \textit{generalized left-symmetric cohomology spaces} with coefficients in  $M$.}
\end{defn}

One can show that $d$ is a homomorphism of  antisymmetric $S$-bimodule, that is, $d(x\cdot f)=x\cdot (df)$ for $x\in S, f\in C^{i}(S,M)$ and  then $Z^{i}(S,M), B^{i}(S,M)$ are antisymmetric $S$-subbimodules, hence $H^{i}(S,M)$ is an antisymmetric $S$-bimodule.

\section{Connections Between Generalized Left-symmetric and  $\epsilon$-Lie Cohomologies}

\setcounter{equation}{0}
\renewcommand{\theequation}
{4.\arabic{equation}}

\setcounter{theorem}{0}
\renewcommand{\thetheorem}
{4.\arabic{theorem}}

Let $\mathfrak{g}$ be an $\epsilon$-Lie algebra, $\mathfrak{M}$ be a left $\mathfrak{g}$-module and
$C(\mathfrak{g},\mathfrak{M})=\oplus_{i\geq 0}C^{i}(\mathfrak{g},\mathfrak{M})$ be the cochain complex on $\mathfrak{g}$ with coboundary operator $\mathfrak{d}$,
where
\begin{equation}
C^{i}(\mathfrak{g},\mathfrak{M})=\textrm{Hom} (\wedge_{\epsilon}^{i} \mathfrak{g},\mathfrak{M})=
\oplus_{\gamma\in\Gamma}\textrm{Hom}_{\gamma}(\wedge_{\epsilon}^{i} \mathfrak{g},\mathfrak{M}).
\end{equation}
Recall that the standard $\epsilon$-Lie algebra representation $\xi$  on $C^i(\mathfrak{g},\mathfrak{M})$ (see \cite{SZ1998})
is given by
\begin{eqnarray*}
(\xi(x)f)(x_{1},\cdots,x_{i})&=&[x, f(x_{1},\cdots,x_{i})]\\
&&- \sum_{j=1}^{i}\epsilon(\alpha,\beta+\alpha_{1}+\cdots+\alpha_{j-1})f(x_{1},\cdots,x_{j-1},[x,x_{j}],\cdots,x_{i}),
\end{eqnarray*}
where $f\in \textrm{Hom}_{\beta}(\wedge_{\epsilon}^{i} \mathfrak{g},\mathfrak{M})$ and $(x,m)\mapsto [x,m]$ is a representation corresponding to the  $\epsilon$-Lie  module $\mathfrak{M}$.

On the other hand, we endow $ C^{i+1}(S,M)$ as defined in (2.19) with a representation $\rho$ of the associated $\epsilon$-Lie algebra $\mathfrak{g}_{S}$ corresponding to the antisymmetric representation constructed in Proposition 2.10:
\begin{eqnarray*}
&& (\rho(x)f)(x_{1},\cdots,x_{i},x_{i+1})\\
&=& x\cdot f(x_{1},\cdots,x_{i},x_{i+1})-\epsilon(\alpha,\beta+\alpha_1+\cdots+\alpha_i)f(x_{1},\cdots,x_{i},x\cdot x_{i+1})\\
&& +\epsilon(\alpha,\beta+\alpha_1+\cdots+\alpha_i) f(x_{1},\cdots,x_{i},x)\cdot x_{i+1}\\
&& -\sum_{s=1}^{i}\epsilon(\alpha,\beta+\alpha_{1}+\cdots+\alpha_{s-1}) f (x_{1},\cdots,x_{s-1},[x,x_{s}],\cdots,x_{i},x_{i+1}),
\end{eqnarray*}
where  $f\in \textrm{Hom}_{\beta}(\wedge_{\epsilon}^{i} S\otimes S,M)$, $x\in (\mathfrak{g}_{S})_{\alpha}$, $x_{s}\in S_{\alpha_{s}}(s=1,\cdots,i,i+1)$.
In particular, the $\mathfrak{g}_{S}$-module structure on $C^1(S,M)$ is given by
\begin{equation*}
[x, f](x_1)=x\cdot f(x_1)-\epsilon(\alpha,\beta)f(x\cdot x_1)+\epsilon(\alpha,\beta)f(x)\cdot x_{1}.
\end{equation*}

\begin{theorem}
Let $S$ be a generalized left-symmetric algebra and $M$ an $S$-bimodule. If we define
$$\psi:C^{i}(\mathfrak{g}_{S},C^{1}(S,M))\rightarrow C^{i+1}(S,M), ~~i>0,$$
by
$$\psi(f)(x_{1},\cdots,x_{i},x_{i+1})=f(x_{1},\cdots,x_{i})(x_{i+1}),$$
then $\psi$ induces an isomorphism of $\mathfrak{g}_{S}$-modules. Moreover, $\psi$ induces an isomorphism of cochain complexes
$\oplus_{i\geq1} C^{i}(S,M)$ and $\oplus_{i\geq0} C^{i}(\mathfrak{g}_{S},C^{1}(S,M))$. In particular, we have
$$H^{i+1}(S,M)\cong H^{i}_{\textrm{Lie}}(\mathfrak{g}_{S}, C^{1}(S,M)) \textrm{ for all } i>0.$$
\end{theorem}

\begin{proof}
We first prove that for all $x\in \mathfrak{g}_{S}$ and $i>0$, the following diagram is commutative:
\[
\begin{CD}
C^{i}(\mathfrak{g}_{S},C^{1}(S,M)) @>{\xi(x)}>>C^{i}(\mathfrak{g}_{S},C^{1}(S,M))\\
@VV{\psi}V @V{\psi}VV \\
C^{i+1}(S,M) @>{\rho(x)}>> C^{i+1}(S,M),
\end{CD}
\] where $\rho$ and $\xi$ are defined as above. In fact, for  $f\in C^{i}_\beta(\mathfrak{g}_{S},C^{1}(S,M))$, we have
\begin{eqnarray*}
&&\psi(\xi(x)f)(x_{1},\cdots,x_{i},x_{i+1})\\
&=&[x, f(x_{1},\cdots,x_{i})](x_{i+1})\\
&&- \sum_{j=1}^{i}\epsilon(\alpha,\beta+\alpha_{1}+\cdots+\alpha_{j-1})f(x_{1},\cdots,x_{j-1},[x,x_{j}],\cdots,x_{i})(x_{i+1})\\
&=&x\cdot f(x_{1},\cdots,x_{i})(x_{i+1})-\epsilon(\alpha,\beta+\alpha_1+\cdots+\alpha_i)f(x_{1},\cdots,x_{i})(x\cdot x_{i+1})\\
&& +\epsilon(\alpha,\beta+\alpha_1+\cdots+\alpha_i) f(x_{1},\cdots,x_{i})(x)\cdot x_{i+1}\\
&& -\sum_{j=1}^{i}\epsilon(\alpha,\beta+\alpha_{1}+\cdots+\alpha_{j-1}) f (x_{1},\cdots,x_{j-1},[x,x_{j}],\cdots,x_{i})(x_{i+1}),
\end{eqnarray*} and
\begin{eqnarray*}
&&\rho(x)(\psi(f))(x_{1},\cdots,x_{i},x_{i+1})\\
&=&x\cdot f(x_{1},\cdots,x_{i})(x_{i+1})-\epsilon(\alpha,\beta+\alpha_1+\cdots+\alpha_i)f(x_{1},\cdots,x_{i})(x\cdot x_{i+1})\\
&& +\epsilon(\alpha,\beta+\alpha_1+\cdots+\alpha_i) f(x_{1},\cdots,x_{i})(x)\cdot x_{i+1}\\
&& -\sum_{j=1}^{i}\epsilon(\alpha,\beta+\alpha_{1}+\cdots+\alpha_{j-1}) f (x_{1},\cdots,x_{j-1},[x,x_{j}],\cdots,x_{i})(x_{i+1}).
\end{eqnarray*}
Thus $\psi:C^{i}(\mathfrak{g}_{S},C^{1}(S,M))\rightarrow C^{i+1}(S,M)$ is a homomorphism
of $\mathfrak{g}_{S}$-modules. Moreover, $\psi$ has no kernel and it is an epimorphism. That is, $\psi$ is an isomorphism.

With an analogous argument, we have that for all $i\geq 0$, the following diagram is commutative:
\[
\begin{CD}
C^{i}(\mathfrak{g}_{S},C^{1}(S,M)) @>{\mathfrak{d}}>>C^{i+1}(\mathfrak{g}_{S},C^{1}(S,M))\\
@VV{\psi}V @V{\psi}VV \\
C^{i+1}(S,M) @>{d}>> C^{i+2}(S,M).
\end{CD}
\]
Hence the cochain complexes $\oplus_{i\geq0}C^{i}(\mathfrak{g}_{S},C^{1}(S,M))$ and $\oplus_{i\geq0}C^{i+1}(S,M)$ are equivalent and in particular, we have
$H^{i+1}(S,M)\cong H^{i}_{\textrm{Lie}}(\mathfrak{g}_{S}, C^{1}(S,M))$ for all $i>0.$
\end{proof}

\begin{rem}{\rm   Note that $H^{1}(S,M)=Z^{1}(S,M)/B^{1}(S,M)$ and
\begin{eqnarray*}
Z^{1}(S,M)&=&\{f\in C^{1}(S,M)| d (f)=0\}=Z^{0}(\mathfrak{g}_{S},C^{1}(S,M)),\\
B^{1}(S,M)&=&\{dm |m\in M, (xy)m=x(ym)\textrm{ for all } x,y \in S\}.
\end{eqnarray*}
Then it is not difficult to check that the following sequence is exact:
$$0\rightarrow Z^{0}(S,M)\rightarrow C^{0}(S,M)\rightarrow H^{0}(\mathfrak{g}_{S},C^{1}(S,M))\rightarrow H^{1}(S,M)\rightarrow 0.$$
}\end{rem}

Theorem 4.1 means that  the cohomology spaces  $H^{i}(S,M)(i\geq 2)$ can be calculated if we are able to find a way to compute
the corresponding $\epsilon$-Lie cohomology spaces. In particular,

\begin{theorem}
For a left-symmetric superalgebra $S$ and its bimodule $M$, we have
$$H^{i+1}(S,M)\cong H^{i}_{Lie}(\mathfrak{g}_{S}, C^{1}(S,M))\cong \sum_{p+q=i}H^p(\mathfrak{g}_{S}/I,k)\otimes H^q(I,C^{1}(S,M))^{\mathfrak{g}_{S}},$$
where $I$ is a graded  ideal of  $\mathfrak{g}_{S}$ such that $\mathfrak{g}_{S}/I$ is a semisimple Lie algebra  and
$H^q(I,C^{1}(S,M))^{\mathfrak{g}_{S}}$ is a submodule of $H^q(I,C^{1}(S,M))$ annihilated by $\mathfrak{g}_{S}.$
\end{theorem}

\begin{proof}
Note that the Hochschild-Serre factorization theorem  in Lie superalgebras (\cite{Bin1986}) and apply Theorem 4.1.
\end{proof}

\section{Deformations of Generalized Left-symmetric Algebras}

\setcounter{equation}{0}
\renewcommand{\theequation}
{5.\arabic{equation}}

\setcounter{theorem}{0}
\renewcommand{\thetheorem}
{5.\arabic{theorem}}

In this section, we extend Gerstenhaber's theory of formal deformation of
algebras (\cite {Ger}) to generalized
left-symmetric algebras.

Let $S = \oplus_{\gamma\in \Gamma}S_\gamma$ be a finite-dimensional
generalized left-symmetric algebra over a field $k$ of characteristic zero  and $k((\lambda))$ denote
the field of fractions for the formal power series ring $k[[\lambda]]$. We extend the coefficient domain from
 $k$ to $k((\lambda))$ and construct a  bilinear map $f_\lambda$ on the vector space $S_\lambda=S\otimes k((\lambda))$ of the form
\begin{equation}
f_{\lambda}(x,y)=x\cdot
y+\lambda F_{1}(x,y)+\lambda^{2}F_{2}(x,y)+\cdots,
\end{equation}
where $F_{i}$ are bilinear maps and we may set $F_{0}(x,y)=x\cdot y$, the product in $S$. Suppose further
that $(x,y)\mapsto f_\lambda(x,y)$ defines a generalized left-symmetric algebra on $S_\lambda$. Then we say that $(S_{\lambda},f_{\lambda})$ is a one-parameter family of deformations of $S$.  It is evident that (2.1) cannot be maintained unless
$f_\lambda$ and
each $F_{i}$ are homogenous elements of degree zero. This is to say, a
deformation must leave the gradation undisturbed.

 On the other hand, the generalized left-symmetric identity (2.2) requires
\begin{equation}
f_{\lambda}(f_{\lambda}(x,y),z)-f_{\lambda}(x,f_{\lambda}(y,z))=\epsilon(\alpha,\beta)\left(f_{\lambda}(f_{\lambda}(y,x),z)-f_{\lambda}(y,f_{\lambda}(x,z))\right)
\end{equation}
for all $x\in S_\alpha,y\in S_\beta,z\in S$  and $\alpha,\beta\in \Gamma.$ In terms of $F_i$, we have  for all non-negative integers
$p$:
\begin{equation}
\sum
_{r+s=p,\\r,s\geq 0}\left(F_{r}(F_{s}(x,y),z)-F_{r}(x,F_{s}(y,z))-\epsilon(\alpha,\beta)\left(F_{r}(F_{s}(y,x),z)-F_{r}(y,F_{s}(x,z))\right)\right)=0.
\end{equation}
 (5.3) are known as integrability conditions.  If $\gamma=1$, we obtain an equation for $F_{1}$:
\begin{eqnarray*}
&&F_{1}(x,y)\cdot z+F_{1}(x\cdot y,z)-x\cdot
F_{1}(y,z)-F_{1}(x,y\cdot
z)\\
&&=\epsilon(\alpha,\beta)(F_{1}(y\cdot
x,z)+F_{1}(y,x)\cdot z-F_{1}(y,x\cdot z)-y\cdot F_{1}(x,z)).\nonumber
\end{eqnarray*}
Hence the first integrability condition states that $F_1$ must be a two-cocycle. An element $F\in Z^2(S,S)$ is said to be integrable if it is
the first term of a one-parameter deformation series (5.1).
We call $F_{1}$ an \textit{infinitesimal deformation}.
Putting $p=2$ in (5.3), we get $$dF_2(x,y,z)=\mu_2(x,y,z),$$
where $$\mu_2(x,y,z)=F_1(F_1(x,y),z)-F_1(x,F_1(y,z))-\epsilon(\alpha,\beta)(F_1(F_1(y,x),z)-F_1(y,F_1(x,z)))$$ and $d$ is the coboundary operator defined as (3.6). One can show that $\mu_2$ is an element of $Z^3(S,S)$ if $F_1$ is in $Z^2(S,S).$ If $F_1$ is integrable, then this
three-cocycle must be a coboundary. Hence the cohomology class of $\mu_2$ is the first obstruction  to the integration of $F_1$.
In general, we have \begin{equation}dF_p(x,y,z)=\mu_p(x,y,z),\end{equation}
where $$\mu_p(x,y,z)=\sum
\left(F_{r}(F_{s}(x,y),z)-F_{r}(x,F_{s}(y,z))-\epsilon(\alpha,\beta)\left(F_{r}(F_{s}(y,x),z)-F_{r}(y, F_{s}(x,z))\right)\right)$$ and
$r+s=p, r,s> 0.$ We can show that if $F_1,\cdots,F_{p-1}$ are chosen such that the integrability condition (5.3) holds, then $\mu_p$
is a three-cocycle. (This is done by appropriate modification of Dzhumadil'daev's proof in \cite{Dzh1999} to take care of the $\Gamma$-gradation.)
Thus, the obstruction to continuing the deformation to  the $p$-th term lies in the
possibility that the cocycle $\mu_p$ might not be a pure coboundary. For this reason,
$H^3(S, S)$ may be regarded as the space of obstruction to deformations of $S$ and if $H^3(S, S)=0$, then all obstructions vanish and
every two-cocycle is integrable.

Now suppose $f_\lambda=\sum\lambda^nF_n$ and $g_\lambda=\sum\lambda^nG_n$ are two one parameter families of
deformations of a generalized left-symmetric algebra $S$. We  say that $f_\lambda$ and $g_\lambda$ are \textit{equivalent}
if there exists a nonsingular linear automorphism $\Phi_\lambda$ of $S_\lambda$ of the form
$$\Phi_\lambda(x) = x + \lambda\varphi_1(x) + \lambda^2\varphi_2(x) + \cdots,$$
where all the $\varphi_i:S\rightarrow S$ are homogenous linear maps of degree zero,
such that for $x,y \in S$
\begin{equation}f_\lambda(x,y)=\Phi^{-1}_\lambda g_\lambda(\Phi_\lambda(x),\Phi_\lambda(y)).\end{equation}
Expanding both sides of (5.5) in powers of $\lambda$, we find
$$F_1(x,y)-G_1(x,y)=d\varphi_1(x,y).$$
Thus, if two deformations are to be equivalent, then their infinitesimal generators must
belong to the same cohomology class in $Z^2(S, S)$.  A deformation $f_\lambda$ is called  \textit{trivial} if it is equivalent to the identity
deformation, that is $F_i=0$ for all $i>0$ in (5.1).

Suppose now that  $f_\lambda$ is a one-parameter family of deformations of $S$ and $F_n (n\geq1)$ is the first nonzero term. Then
it follows from (5.4) that $dF_n=0$. If further $F_n$ is in $B^2(S,S),$  then $F_n=-d\varphi$ for some $\varphi$  in $C^1(S,S).$ Setting
$\Phi_\lambda(x)=x+\lambda\varphi(x),$ we have
$$\Phi_\lambda^{-1}f_\lambda(\Phi_\lambda(x),\Phi_\lambda(y))=x\cdot y+\lambda^{n+1}F'_{n+1}(x,y)+\lambda^{n+2}F'_{n+2}(x,y)+\cdots,$$
where $F'_{n+1},F'_{n+2},\cdots,$ are the two-cochains of degree zero defining a deformation $f'_\lambda$ that is equivalent to $f_\lambda$ and again $F'_{n+1}$ is in $Z^2(S,S).$ Then we obtain the following result.

\begin{prop}
Let $f_\lambda$ be a one-parameter family of deformations of a generalized left-symmetric algebra $S$. Then $f_\lambda$ is equivalent to
$g_\lambda(x,y)=x\cdot y+\lambda^nG_n(x,y)+\cdots,$ where the first non-vanishing cochain $G_n$ is in $Z^2(S,S)$ and is not a pure coboundary.
In particular, if $H^2(S,S)=0$, then every deformation is equivalent to the trivial deformation.
\end{prop}

The study of  deformations
is a way to obtain new algebras. We end this paper with an example
that illustrates  how to determine all simple  left-symmetric superalgebras which can be obtained
by infinitesimal deformations of a given left-symmetric superalgebra.

\begin{exam}{\rm
Consider the  3-dimensional complex left-symmetric superalgebra $S$ with
a homogeneous basis $\{x,y_{1},y_{2}~|~ x\in
S_{\bar{0}},y_{1},y_{2}\in S_{\bar{1}}\}$ satisfying
$$x\cdot x=2x,x\cdot y_{1}=y_{1},x\cdot y_{2}=y_{2},y_{1}\cdot y_{2}=x,y_{2}\cdot
y_{1}=-x.$$
With respect to the basis of $S$ the second cohomology is given by
\begin{eqnarray*}
H^2(S,S)=\{F_1\in C^2(S,S)| F_1(x,x)=ax, F_1(x,y_1)=by_2,
\\ F_1(x,y_2)=cy_1+ay_2,~~ \textrm{others are zero,} ~~a,b,c\in \textbf{C}\}.
\end{eqnarray*}
These two-cocycles satisfy the integrability conditions (5.3).
We have a first-order deformation with $$f_\lambda=F_0+\lambda F_1.$$
Next, we determine  the simple left-symmetric
superalgebra  $(S_\lambda,f_\lambda).$
Note that every nonzero ideal will contain $x$. On the other hand, the ideal generated by $x$ equals $S_\lambda$ if the vectors
$f_\lambda(x, y_1), f_\lambda(x,y_2)$ are linear independent.
Then it is not difficult to classify the simple left-symmetric
superalgebras  constructed by this way: any one of them is
isomorphic to one of the following left-symmetric superalgebras
\begin{enumerate}
  \item ${S}_{\lambda_{1,t}}: f_\lambda(x, x)=(t+1)x, f_\lambda(x, y_{1})=y_{1},
f_\lambda(x, y_{2})=ty_{2}, f_\lambda(y_{1}, y_{2})=x, f_\lambda(y_{2}, y_{1})=-x,0<|t|<1\textrm{ or }t=e^{i\theta},0\leq\theta\leq\pi;$
  \item ${S}_{\lambda_{2}}:f_\lambda(x, x)=2x, f_\lambda(x, y_{1})=y_{1}, f_\lambda(x,
y_{2})=y_{1}+y_{2}, f_\lambda(y_{1}, y_{2})=x, f_\lambda(y_{2}, y_{1})=-x$.
\end{enumerate}
Here ${S}_{\lambda_{1,t}}$ and ${S}_{\lambda_{2}}$ contain all the  3-dimensional complex simple left-symmetric
superalgebras which have been obtained in \cite{ZB2012}. }\end{exam}

\section*{\textit{Acknowledgments}}
This work was supported by NNSF of China (11226051) and the Fundamental Research
Funds for the Central Universities (11QNJJ001).

\end{document}